\documentclass[12pt]{article}
\usepackage{fullpage,amsthm,amssymb,amsmath, amsfonts}
\usepackage{graphicx,algorithmic,algorithm}
\usepackage{enumerate}
\usepackage{tikz}
\usetikzlibrary{shapes}
\usepackage[normalem]{ulem}

\newtheorem{theorem}{Theorem}[section]
\newtheorem{lemma}[theorem]{Lemma}
\newtheorem{proposition}[theorem]{Proposition}
\newtheorem{corollary}[theorem]{Corollary}
\newtheorem{claim}[theorem]{Claim}
\newtheorem{question}[theorem]{Question}
\theoremstyle{definition}

\renewcommand\sout{}  

\begin{document}
\title{Toroidal graphs containing neither $K_5^{-}$ nor $6$-cycles are $4$-choosable}
\author{Ilkyoo Choi\thanks{Department of Mathematics, University of Illinois at Urbana-Champaign. \texttt{ichoi4@illinois.edu}} }
\date\today
\maketitle
\begin{abstract}
The choosability $\chi_\ell(G)$ of a graph $G$ is the minimum $k$ such that having $k$ colors available at each vertex guarantees a proper coloring. 
Given a toroidal graph $G$, it is known that $\chi_\ell(G)\leq 7$, and $\chi_\ell(G)=7$ if and only if $G$ contains $K_7$.
Cai, Wang, and Zhu proved that a toroidal graph $G$ without $7$-cycles is $6$-choosable, and $\chi_\ell(G)=6$ if and only if $G$ contains $K_6$. 
They also prove that a toroidal graph $G$ without $6$-cycles is $5$-choosable, and conjecture that $\chi_\ell(G)=5$ if and only if $G$ contains $K_5$. 
We disprove this conjecture by constructing an infinite family of non-$4$-colorable toroidal graphs with neither $K_5$ nor cycles of length at least $6$; moreover, this family of graphs is embeddable on every surface except the plane and the projective plane. 
Instead, we prove the following slightly weaker statement suggested by Zhu: toroidal graphs containing neither $K^-_5$ (a $K_5$ missing one edge) nor $6$-cycles are $4$-choosable. 
This is sharp in the sense that forbidding only one of the two structures does not ensure that the graph is $4$-choosable. 
\end{abstract}

\section{Introduction}

Let $[n]=\{1, \ldots, n\}$. 
Only finite, simple graphs are considered. 
Let $K_n$ be the complete graph on $n$ vertices. 
If $H$ is a subgraph of $G$, then we write $H\subseteq G$.
Given a graph $G$, let $V(G)$ and $E(G)$ denote the vertex set and edge set of $G$, respectively. 
Given a graph $G$, a {\it list assignment} $L$ is a function on $V(G)$ that assigns to each vertex $v$ a list $L(v)$ of (\emph{available}) colors.
An \emph{$L$-coloring} is a vertex coloring $f$ such that $f(v) \in L(v)$ for each vertex $v$ and $f(x)\neq f(y)$ for each edge $xy$.
A graph $G$ is said to be {\it $k$-choosable} if there is an $L$-coloring for each list assignment $L$ where $|L(v)|\geq k$ for each vertex $v$. 
The minimum such $k$ is known as the {\it choosability} of $G$, denoted $\chi_\ell(G)$.

Thomassen \cite{1994Th} proved that planar graphs are $5$-choosable, and Voigt \cite{1993Vo} constructed a planar graph that is not $4$-choosable. 
It is known that \cite{1999LaXuLi,2002WaLi,2001WaLi,2009Fa} that planar graphs without $k$-cycles for some $k\in\{3, 4, 5, 6, 7\}$ are $4$-choosable.
There is also a vast literature on forbidding cycles in a planar graph to ensure that it is $3$-choosable; we direct the readers to \cite{2013Bo} for a very thorough survey.

Regarding toroidal graphs, B\"ohme, Mohar, and Stiebitz \cite{1999BoMoSt} showed that they are $7$-choosable, and they characterized when the choosability of a toroidal graph is $7$ by proving that a toroidal graph $G$ has $\chi_\ell(G)=7$ if and only if $K_7\subseteq G$. 
Cai, Wang, and Zhu \cite{2010CaWaZh} establish several tight results on the choosabiltiy of a toroidal graph with no short cycles. 
In particular, they prove that a toroidal graph $G$ with no $7$-cycles is $6$-choosable, and $\chi_\ell(G)=6$ if and only if $K_6\subseteq G$. 
They also prove that a toroidal graph with no $6$-cycles is $5$-choosable, and conjecture that $\chi_\ell(G)=5$ if and only if $K_5\subseteq G$. 

We disprove this conjecture by constructing an infinite family of toroidal graphs containing neither $K_5$ nor $6$-cycles that is not even $4$-colorable. (See Theorem~\ref{noC6}.)
It is worth mentioning that this infinite family of graphs is embeddable on any surface, orientable or non-orientable, except the plane and the projective plane. 
This shows that for the family of graphs embeddable on some surface, forbidding a cycle of length $6$ and $K_5$ is not enough to ensure $4$-choosability for any surface except the plane and the projective plane. 
Recall that forbidding a cycle of length $6$ is enough to ensure $4$-choosability in a planar graph.
Therefore, the following question by Kostochka~\cite{00Ko} is natural:

\begin{question}
Is every projective plane graph containing neither $K_5$ nor $6$-cycles $4$-choosable?
\end{question}

The main result of this paper is a different weakening of the original conjecture suggested by Zhu~\cite{00Zh}: 

\begin{theorem}\label{result}
A toroidal graph containing neither $K^-_5$ nor $6$-cycles is $4$-choosable. 
\end{theorem}

This theorem is sharp in the sense that forbidding only one of a $K^-_5$ or $6$-cycles in a toroidal graph does not guarantee that it is $4$-choosable. 

In the figures throughout this paper, the white vertices do not have incident edges besides the ones drawn, and the black vertices may have other incident edges. 

\begin{figure}[h]
\centering
\includegraphics{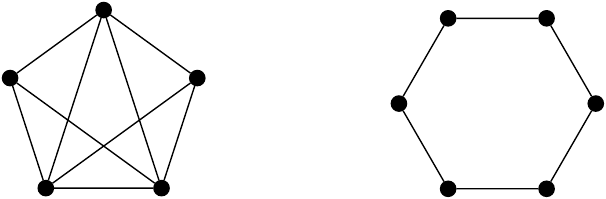}
\caption{Forbidden configurations.
}\label{fig:tikz:forbidden}
\end{figure}

In section $2$, we prove many structural lemmas needed in Section $3$, which is where we prove Theorem~\ref{result} using discharging. 
In Section $4$, we display the sharpness examples of Theorem~\ref{result}.

\section{Lemmas}

From now on, let $G$ be a counterexample to Theorem~\ref{result} with the fewest number of vertices, and fix some embedding of $G$.
It is easy to see that the minimum degree of (a vertex of) $G$ is at least $4$ and $G$ is connected.

The {\it neighborhood} of a vertex $v$, denoted $N(v)$, is the set of vertices adjacent to $v$, and let 
$N[v]=N(v)\cup\{v\}$.
The {\it degree} of a vertex $v$, denoted $d(v)$, is $|N(v)|$.
The {\it degree} of a face $f$, denoted $d(f)$, is the length of $f$. 
A {\it $k$-vertex}, {\it $k^+$-vertex}, {\it $k$-face}, {\it $k^+$-face} is a vertex of degree $k$, a vertex of degree at least $k$, a face of degree $k$, and a face of degree at least $k$, respectively. 

A graph is {\it degree-choosable} if there is an $L$-coloring for each list assignment $L$ where $|L(v)|\geq d(v)$ for each vertex $v$. 
The following is a very well-known fact.

\begin{theorem}\label{deg-choos}
A graph is degree-choosable unless each maximal $2$-connected subgraph is either a complete graph or an odd cycle. 
\end{theorem}

A set $S\subseteq V(G)$ of vertices is {\it $k$-regular} if every vertex in $S$ has degree $k$ in $G$.
A {\it chord} is an edge between two non-consecutive vertices on a cycle. 
Let $W_4$ be a $K_5$ missing two edges that are not incident to each other.

\begin{lemma}\label{reducible}
$V(G)$ does not contain any of the following:
\begin{enumerate}[$(i)$]
\item A $4$-regular set $S$ where $G[S]$ is a cycle of even length.
\item A $4$-regular set $S$ where $G[S]$ is a cycle with one chord. 
\item A set $S$ with four vertices of degree $4$ and one vertex of degree $5$ where $G[S]$ is $W_4$. 
\item A set $S$ where $G[S]$ is a $5$-face with one vertex of degree $1$.
\item A set $S$ where $G[S]$ is a $6$-face with one vertex of degree $1$.
\end{enumerate}
\end{lemma}
\sout{\begin{proof}
Assume for the sake of contradiction that $G$ contains a $4$-regular set $S$ described in either $(i)$ or $(ii)$. 
By the minimality of $G$, there exists an $L$-coloring $f$ of $G-S$. 
For $v\in S$, let $L_f(v)=L(v)\setminus\{f(u): u\in N(v)\setminus S\}$. 
By Lemma~\ref{deg-choos}, since $|L_f(v)|$ is at least the degree of $v$ in $G[S]$, it follows that there exists an $L_f$-coloring $g$ of $G[S]$.
By combining $f$ and $g$, we obtain an $L$-coloring of $G$, which contradicts that $G$ is a counterexample. 
$(iii)$ follows from $(ii)$ since $(iii)$ contains $(ii)$ as a subgraph. 
$(iv)$ and $(v)$ also cannot exist since $G$ has minimum degree at least $4$. 
\end{proof}
}

\begin{figure}[h]
\centering
\includegraphics{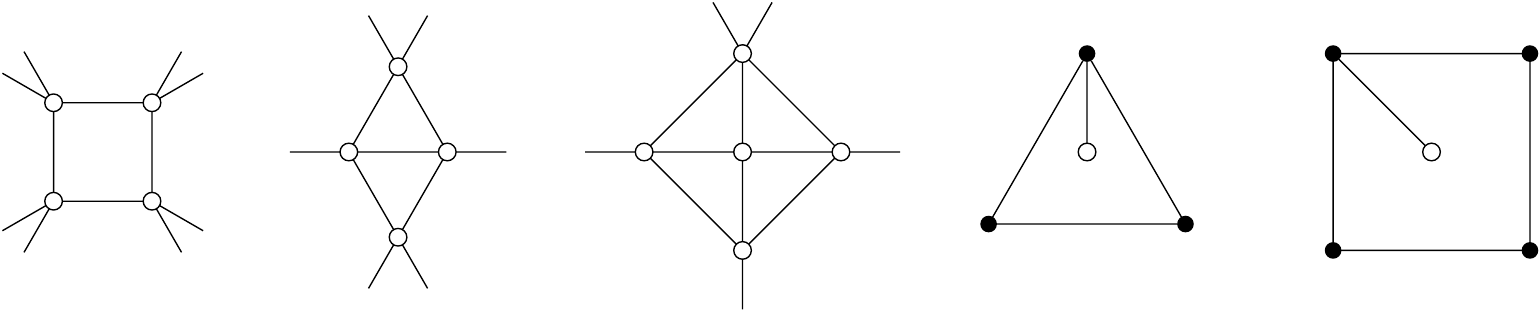}
  \caption{Forbidden configurations of $G$. 
}
  \label{fig:tikz:reducible}
\end{figure}

A $6$-face is {\it degenerate} if some vertex $y$ is incident to it twice; namely, it is of the form $xyzayw$ (see Figure~\ref{fig:tikz:degenerate}).
A list of faces of a vertex $v$ is {\it consecutive} if it is a sublist of the list of faces incident to $v$ in cyclic order. 

\begin{proposition}\label{degen6}
If $f$ is a $6$-face of $G$ where $wxyz$ are consecutive vertices on $f$, then the following holds:
\begin{enumerate}[$(i)$]
\item $f$ must be a degenerate $6$-face.
\item If $xz$ is not an edge, then $wy$ is an edge and $y$ is incident to $f$ twice. 
\item If $w\neq z$, then either $x$ or $y$ is incident to $f$ twice. 
\item $f$ cannot appear consecutively in the list of consecutive faces of a vertex. 
\end{enumerate}
\end{proposition}
\sout{\begin{proof}
$(i)$ follows from Lemma~\ref{reducible} $(v)$. 
It is easy to check $(ii)$, $(iii)$, and $(iv)$. 
\end{proof}
}

\begin{figure}[h]
\centering
\includegraphics{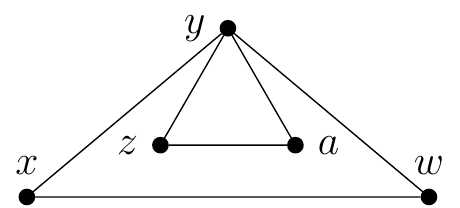}
\caption{A degenerate $6$-face.
}
\label{fig:tikz:degenerate}
\end{figure}

\begin{proposition}\label{adj4faces}
Given a $4$-face $vu_2xu_3$ and $u_1\not\in\{v, u_2, u_3, x\}$, if $u_1vu_2y$ is a $4$-face for some vertex $y$, then $y=u_3$. 
\end{proposition}
\sout{
\begin{proof}
Note that $y\not\in\{v, u_1, u_2\}$, and if $y=x$, then $d(u_2)=2<4$, which contradicts the minimum degree of $G$. 
Now, $vu_1yu_2xu_3$ is a $6$-cycle, unless $y=u_3$. 
\end{proof}
}

\begin{claim}\label{faces303}
If $f_1, f_2, f_3$ are consecutive faces of a vertex $v$ with $d(f_1)=d(f_3)=3\neq d(f_2)$, then $d(f_2)\geq 6$. 
\end{claim}
\sout{
\begin{proof}
Let $u_1, u_2, u_3, u_4$ be neighbors of $v$ in cyclic order so that $f_1$ is $vu_1u_2$ and $f_3$ is $vu_3u_4$. 
If $f_2$ is a $4$-face $u_2vu_3x$, then $vu_1u_2xu_3u_4$ is a $6$-cycle, unless $x\in \{u_1, u_4\}$.
Yet, if $x\in \{u_1, u_4\}$, then either $d(u_2)=2$ or $d(u_3)=2$, which contradicts the minimum degree of $G$. 
If $f_2$ is a $5$-face $u_2vu_3xy$, then $G$ has a $6$-cycle, unless $\{x, y\}= \{u_1, u_4\}$.
If $x=u_4$ or $y=u_1$, then either $d(u_3)=2$ or $d(u_2)=2$. 
Thus, $x=u_1$ and $y=u_4$, which implies $u_1u_3, u_1u_4, u_2u_4\in E(G)$. 
Yet, now $K^-_5\subseteq G[N[v]]$.
\end{proof}
}


\begin{claim}\label{faces3340}
If $f_1, f_2, f_3, f_4$ are consecutive faces of a vertex $v$ with $d(f_1)=d(f_2)=3$ and $d(f_3)=4$, then $d(f_4)\geq 6$. 
\end{claim}
\sout{
\begin{proof}
Let $u_1, u_2, u_3, u_4$ be the neighbors of $v$ in cyclic order so that $f_1$ is $vu_1u_2$, $f_2$ is $vu_2u_3$, and $f_3$ is $u_3vu_4x$ for some $x$.
If $x\not\in \{u_1, u_2\}$, then $u_1u_2u_3xu_4v$ is a $6$-cycle, which is a contradiction.
If $x=u_2$, then $d(u_3)=2$, which is a contradiction.
Therefore, $x=u_1$. 

Note that if either $u_3u_4$ or $u_2u_4$ is an edge, then $K^-_5\subseteq G[N[v]]$. 
Also, $vu_4$ and $u_4u_1$ cannot be consecutive edges on the boundary of $f_4$ since this implies $d(u_4)=2$. 
If $f_4$ is a $3$-face $vu_4x$, then $x\not\in\{u_1, u_2, u_3\}$. 
Yet, $vxu_4u_1u_3u_2$ is a $6$-cycle. 
If $f_4$ is a $4$-face $vu_4xy$, then $x\not\in\{u_1, u_2, u_3\}$.
If $y\not\in \{u_1, u_2, u_3\}$, then $vyxu_4u_1u_2$ is a $6$-cycle. 
If $y=u_1$, then $vu_4xyu_2u_3$ is a $6$-cycle. 
If $y=u_2$, then $u_4xyu_3vu_1$ is a $6$-cycle. 
If $y=u_3$, then $u_4xyvu_2u_1$ is a $6$-cycle. 
If $f_4$ is a $5$-face $vu_4xyz$, then $x, y\not\in\{u_1, u_2, u_3\}$. 
If $z\not\in\{u_1, u_2, u_3\}$, then $u_4xyzvu_1$ is a $6$-cycle. 
If $z=u_1$, then $u_4xyzu_2v$ is a $6$-cycle.
If $z=u_2$, then $u_4xyzvu_1$ is a $6$-cycle.
If $z=u_3$, then $u_4xyzvu_1$ is a $6$-cycle. 
\end{proof}
}

\begin{corollary}\label{faces3303}
If $f_1, f_2, f_3, f_4$ are consecutive faces of a $5$-vertex $v$ with $d(f_1)=d(f_2)=3$ and $d(f_4)=3$, then $d(f_3)\geq 7$. 
\end{corollary}
\sout{\begin{proof}
Let $u_1, u_2, u_3, u_4, u_5$ be the neighbors of $v$ in cyclic order so that $f_1$ is $u_1vu_2$, $f_2$ is $u_2vu_3$, and $f_4$ is $u_4vu_5$.
By Claim~\ref{faces303}, $d(f_3)\geq 6$. 
Assume for the sake of contradiction that $d(f_3)=6$. 
If $u_3u_4$ is not an edge, then by Proposition~\ref{degen6} $(ii)$, $v$ must be incident to $f_3$ twice. 
This implies that $f_3$ is either $u_3vu_4u_5vu_1$ or $u_3vu_4u_1vu_5$. 
In the former, $d(u_4)=2$, and in the latter, $u_3u_5u_4u_1u_2v$ is a $6$-cycle.
\end{proof}
}

\begin{claim}\label{faces4434}
There is no $5$-vertex $v$ with $d(f_1)=d(f_2)=d(f_4)=4$ and $d(f_3)=3$  where $f_1, f_2, f_3, f_4$ are consecutive faces of $v$. 
\end{claim}
\sout{
\begin{proof}
Let $u_1, u_2, u_3, u_4, u_5$ be the neighbors of $v$ in cyclic order so that $f_1$ is $u_1vu_2x$, $f_2$ is $u_2vu_3y$, and $f_3$ is $u_3vu_4$, and $f_4$ is $u_4vu_5z$ for some $x, y, z$. 
Note that $y\neq u_4$ since otherwise $d(u_3)=2$, and $z\neq u_3$ since otherwise $d(u_4)=2$. 

Assume $y\not\in \{u_1, u_5\}$. 
By considering $f_1$ and $f_2$ and Proposition~\ref{adj4faces}, $x=u_3$. 
If $z\not\in\{u_1, u_3\}$, then $u_4zu_5vu_1u_3$ is a $6$-cycle. 
Thus, $z=u_1$. 
Yet, now $u_4u_1vu_2yu_3$ is a $6$-cycle. 

Assume $y=u_1$.
If $z\not\in\{u_1, u_3\}$, then $u_4zu_5vu_1u_3$ is a $6$-cycle. 
Thus, $z=u_1$. 
If $x\not\in\{u_3, u_4\}$, then $u_4u_1xu_2vu_3$ is a $6$-cycle. 
Yet, if $x=u_3$, then $u_4vu_5u_1u_2u_3$ is a $6$-cycle, and if $x=u_4$, then $u_4u_3vu_5u_1u_2$ is a $6$-cycle. 

Assume $y=u_5$.
If $z\not\in\{u_2, u_3\}$, then $u_4zu_5u_2vu_3$ is a $6$-cycle.
Thus, $z=u_2$. 
If $x\not\in\{u_3, u_4\}$, then $u_1xu_2u_4u_3v$ is a $6$-cycle.
Yet, if $x=u_3$, then $u_1u_3u_4u_2u_5v$ is a $6$-cycle, and if $x=u_4$, then $u_1u_4u_3u_5u_2v$ is a $6$-cycle.
\end{proof}
}

\begin{figure}[h]
\centering
\includegraphics{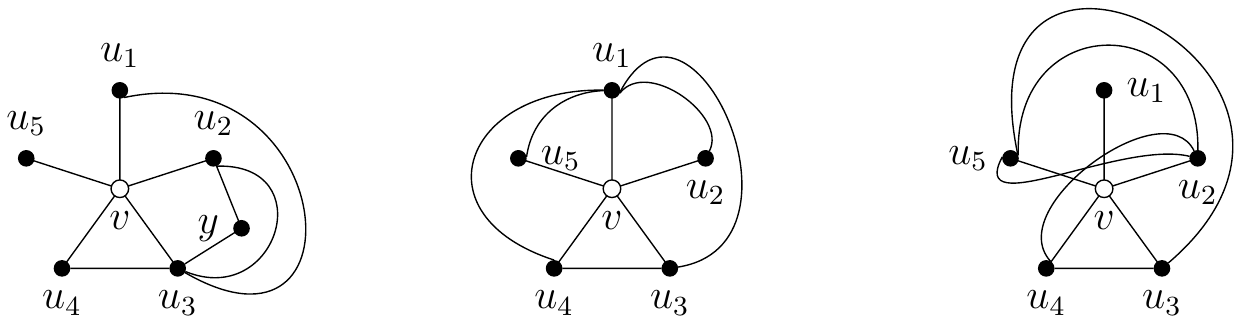}
  \caption{Cases for Claim~\ref{faces4434}.
}\label{fig:tikz:fig-faces4434}
\end{figure}

\begin{claim}\label{faces44444}
There is no $5$-vertex $v$ that is incident to only $4$-faces. 
\end{claim}
\sout{
\begin{proof}
Let $u_1, u_2, u_3, u_4, u_5$ be the neighbors of $v$ in cyclic order so that $u_4vu_5x$ is a $4$-face for some $x$. 

Assume $x\not\in \{u_1, u_2, u_3\}$. 
By considering the two $4$-faces adjacent to $vu_4$ and Proposition~\ref{adj4faces}, $u_3u_5, u_4u_5\in E(G)$. 
By considering the two $4$-faces adjacent to $vu_5$ and Proposition~\ref{adj4faces}, $u_1u_4\in E(G)$. 
Now, $u_1u_4xu_5u_3v$ is a $6$-cycle.

Assume $x=u_2$. 
By considering the two $4$-faces adjacent to $vu_5$ and Proposition~\ref{adj4faces}, $u_4u_5, u_4u_1\in E(G)$. 
By considering the two $4$-faces adjacent to $vu_4$ and Proposition~\ref{adj4faces}, $u_3u_5\in E(G)$. 
Now, $vu_1u_4u_2u_5u_3$ is a $6$-cycle. 

The only cases left are $x\in\{u_3, u_1\}$.
Without loss of generality, assume $x=u_3$.
By considering the two $4$-faces adjacent to $vu_5$ and Proposition~\ref{adj4faces}, $u_4u_5, u_4u_1\in E(G)$. 
By considering the two $4$-faces adjacent to $vu_1$ and Proposition~\ref{adj4faces}, $u_5u_1, u_5u_2\in E(G)$. 
Yet, $vu_2u_5u_1u_4u_3$ is a $6$-cycle. 
\end{proof}
}

A $4$-vertex $v$ is {\it special} if $v$ is incident to a $4$-face and exactly two $3$-faces.

\begin{corollary}
The two $3$-faces incident to a special vertex $v$ must be adjacent to each other. 
\end{corollary}
\sout{
\begin{proof}
If the two $3$-faces are nonadjacent, then Claim~\ref{faces303} says no $4$-face is incident to $v$. 
\end{proof}
}

\begin{figure}[h]
\centering
\includegraphics{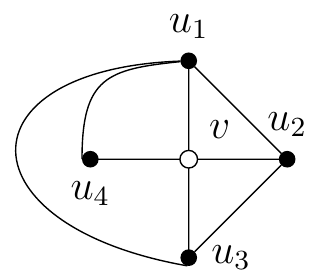}
  \caption{A special vertex.
}\label{fig:tikz:fig-special}
\end{figure}

\begin{claim}\label{4face-oneSpecial}
Each $4$-face is incident to at most one special vertex.
\end{claim}
\sout{
\begin{proof}
Let $u_1, u_2, u_3, u_4$ be the neighbors of a special vertex $v$ in cyclic order so that $vu_1u_2$ and $vu_2u_3$ are the two $3$-faces incident to $v$, and $u_3vu_4x$ is a $4$-face for some $x$.
If $x\not\in \{u_1, u_2\}$, then $u_1u_2u_3xu_4v$ is a $6$-cycle.
If $x=u_2$, then $d(u_3)=2$.
Therefore, $x=u_1$. 

Note that if either $u_3u_4$ or $u_2u_4$ is an edge, then $K^-_5\subseteq G[N[v]]$. 
If $u_1$ is a special vertex, then $u_1u_2x$ must be a $3$-face for some $x\in \{u_3, u_4\}$, otherwise $u_1xu_2u_3vu_4$ is a $6$-cycle. 
Since $x=u_4$ creates a $K^-_5$, it must be that $x=u_3$, but this implies that $d(u_2)=3$.
If $u_3$ is a special vertex, then $u_2u_3x$ must be a $3$-face for some $x\in \{u_1, u_4\}$, otherwise $u_2xu_3u_1u_4v$ is a $6$-cycle.
Since $x=u_4$ creates a $K^-_5$, it must be that $x=u_1$, but this implies that $d(u_2)=3$.
If $u_4$ is a special vertex, then since $vu_4u_1$ cannot be a $3$-face, it must be that $u_4u_1x$ is a $3$-face for some $x\in \{u_2, u_3\}$, otherwise $u_1xu_4vu_3u_2$ is a $6$-cycle. 
Yet either choice of $x$ creates a $K^-_5$. 
Hence none of $u_1, u_3, u_4$ can be a special vertex, and thus there is only at most special vertex.
\end{proof}
}

\begin{claim}\label{faces3454}
If $f_1, f_2, f_3, f_4$ are consecutive faces of a $4$-vertex $v$ with $d(f_1)=3$, $d(f_2)=d(f_4)=4$ and $d(f_3)\geq 5$, then neither $f_2$ nor $f_4$ is incident to a special vertex.
\end{claim}
\sout{
\begin{proof}
Let $u_1, u_2, u_3, u_4$ be the neighbors of $v$ in cyclic order so that $vu_1u_2$ is the $3$-face $f_1$ incident to $v$. 
Let the $4$-face $f_2$ be $u_2vu_3x$ and let the other $4$-face $f_4$ be $u_1vu_4y$. 
If $x=u_1$, then $d(u_2)=2$, which is a contradiction.
If $y=u_2$, then $d(u_1)=2$.
If $x=u_4$ and $y\not\in N[v]$, then $u_1u_2vu_3u_4y$ is a $6$-cycle.
If $y=u_3$ and $x\not\in N[v]$, then $u_2u_1vu_4u_3x$ is a $6$-cycle.
So either $x=u_4$ and $y=u_3$ or $x, y\not\in N[v]$. 
Note that $v$ cannot be a special vertex since it is incident to two $4$-faces. 

Assume $x=u_4$ and $y=u_3$. 
If $u_2$ is a special vertex, then $u_1u_2z$ must be a $3$-face for some $z\neq v$.
If $z\not\in\{u_3, u_4\}$, then $zu_1vu_3u_4u_2$ is a $6$-cycle. 
If $z=u_3$, then $K^-_5\subseteq G[N[v]]$. 
If $z=u_4$, then $d(u_2)=3$. 
Note that $u_3, u_4$ are not special vertices since each is incident to two $4$-faces. 
Therefore, the $f_2$ is not incident to a special vertex, and by similar logic, $f_4$ is not incident to a special vertex. 

Assume $x, y\not\in N[v]$.
If $x=y$, then $x$ cannot be a special vertex since it is incident to two $4$-faces. 
Without loss of generality, assume $u_1u_2z$ is a $3$-face for some $z\neq v$. 
If $z\not\in\{x, u_3\}$, then $zu_1vu_3xu_2$ is a $6$-cycle. 
If $z=x$, then $d(u_2)=3$.
If $z=u_3$, then $K^-_5\subseteq G[N[v]\cup\{x\}]$.
Since $u_1u_2z$ cannot be a $3$-face, it follows that both $u_2$ and $u_1$ cannot be special vertices. 
If $u_3$ is a special vertex, then $u_3xz$ must be a $3$-face for some $z\neq v$. 
If $z\not\in\{u_1, u_2\}$, then $u_1u_2xzu_3v$ is a $6$-cycle. 
If $z=u_1$, then $u_1u_2xu_4vu_3$ is a $6$-cycle. 
If $z=u_2$, then $u_2u_1vu_4xu_3$ is a $6$-cycle. 
Therefore, neither $f_2$ nor $f_4$ is incident to a special vertex.

If $x\neq y$, then both $u_1, u_2$ cannot be special vertices since $u_1u_2z$ cannot be a $3$-face for some $z\neq v$; this is because if $z\not\in\{x, u_3\}$ then $vu_1zu_2xu_3$ is a $6$-cycle, and if $z\not\in\{y, u_4\}$ then $u_1zu_2vu_4y$ is a $6$-cycle. 
If $xu_2z$ is a $3$-face for some $z$, then $z\in\{v, u_1, u_3\}$, otherwise $zu_2u_1vu_3x$ is a $6$-cycle. 
If $z=u_1$, then $d(u_2)=3$, and if $z=u_3$ then $d(x)=2$. 
If $z=v$, then $zxu_2u_1yu_4$ is a $6$-cycle. 
If $xu_3z$ is a $3$-face for some $z$, then $z\in\{u_1, u_2, v\}$, otherwise, $u_1u_2xzu_3v$ is a $6$-cycle. 
If $z\in\{v, u_2\}$, then either $d(u_3)=2$ or $d(u_2)=3$. 
If $z=u_1$, then $u_1yu_4vu_3x$ is a $6$-cycle. 
Therefore, $f_2$ is not incident to a special vertex, and by similar logic, $f_4$ is also not incident to a special vertex. 
\end{proof}
}

%

\begin{claim}\label{4vertex-304}
If $f_1, f_2, f_3, f_4$ are consecutive faces of a non-special $4$-vertex $v$ where $d(f_1)=3$ and $d(f_3)=4$, then one of the following holds:
\begin{enumerate}[$(i)$]
\item $d(f_i)\geq 6$ and $d(f_j)\geq 5$ where $\{i, j\}=\{2, 4\}$;
\item $d(f_i)\geq 6$ and $d(f_j)=4$ and $f_3$ is not incident to a special vertex where $\{i, j\}=\{2, 4\}$.
\end{enumerate}
\end{claim}
\sout{
\begin{proof}
Let $u_1, u_2, u_3, u_4$ be the neighbors of $v$ in cyclic order so that $f_1$ is $vu_1u_2$ and $f_3$ is $u_4vu_3x$.
Assume $x\not\in\{u_1, u_2\}$. 
Consider the face $f_2$. 
If $f_2$ is a $3$-face, then $u_1u_2u_3xu_4v$ is a $6$-cycle.
If $f_2$ is a $4$-face $u_2vu_3y$, then by Proposition~\ref{adj4faces}, $y=u_4$.
Yet, now $vu_1u_2u_4xu_3$ is a $6$-cycle. 
If $f_2$ is a $5$-face $u_2vu_3yz$, then $vu_1u_2zyu_3$ is a $6$-cycle, unless $u_1\in\{z, y\}$.
If $u_1=z$, then $d(u_2)=2$.
If $u_1=y$, then $vu_2u_1yxu_4$ is a $6$-cycle. 
Therefore, $d(f_2)\geq 6$, and by symmetry, $d(f_4)\geq 6$. 

Without loss of generality, assume $x=u_2$ and consider $f_4$. 
Note that $f_2$ cannot be a $3$-face since this implies that $d(u_3)=2$. 
Since $v$ is not special, this implies that $f_4$ cannot be a $3$-face. 
If $f_4$ is a $4$-face $u_1vu_4y$, then by Proposition~\ref{adj4faces}, $y=u_3$. 
Yet, now $K^-_5\subseteq G[N[v]]$. 
If $f$ is a $5$-face $u_1vu_4yz$, then $u_1u_2vu_4yz$ is a $6$-cycle, unless $u_2\in\{y, z\}$.
If $u_2=z$, then $d(u_1)=2$, and if $u_2=y$, then $d(u_4)=2$. 
Therefore, $d(f_4)\geq 6$. 
If $d(f_2)\geq 5$, then $(i)$ is satisfied.
If $d(f_2)=4$, then $(ii)$ is satisfied since $u_2, u_3, u_4$ are each incident to at least two $4$-faces, none of them can be special.
\end{proof}
}

\begin{figure}[h]
\centering
\includegraphics{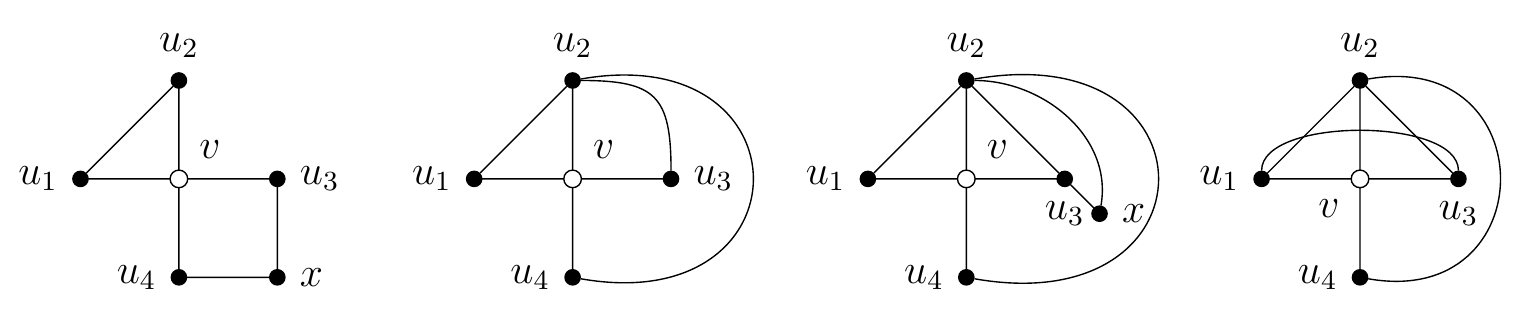}
  \caption{Pictures for Claim~\ref{4vertex-304} and Claim~\ref{4vertex-335}. 
}\label{fig:tikz:fig-helper-4vertex}
\end{figure}

\begin{claim}\label{4vertex-335}
If $f_1, f_2, f_3, f_4$ are consecutive faces of a $4$-vertex $v$ where $d(f_1)=d(f_2)=3$, $d(f_3)=5$, and $d(f_4)\geq 5$, then $d(f_4)\geq 7$ and $f_3$ is incident to a $5^+$-vertex. 
\end{claim}
\sout{
\begin{proof}
Let $u_1, u_2, u_3, u_4$ be the neighbors of $v$ in cyclic order so that $f_1$ is $u_1vu_2$, $f_2$ is $u_2vu_3$, and $f_3$ is $u_4vu_3xy$ for some $x, y$. 
Note that if $x=u_2$, then $d(u_3)=2$. 

Assume $x=u_1$.
If $d(u_1)=d(u_2)=d(u_3)=4$, then this contradicts Lemma~\ref{reducible} $(ii)$.
Thus, some vertex has higher degree, and therefore $f_3$ is incident to a $5^+$-vertex.
If $f_4$ is a $6$-face, then since $v$ cannot be incident to $f_4$ twice, it must be that $u_1u_4$ is an edge by Proposition~\ref{degen6}.
Yet, $K^-_5\subseteq G[N[v]]$. 
If $f_4$ is a $5$-face $zu_1vu_4w$, then $u_1u_2vu_4wz$ is a $6$-cycle, unless $u_2\in\{z, w\}$. 
If $u_2=z$, then $d(u_1)=2$ and if $u_2=w$, then $d(u_4)=2$. 

Assume $x\not\in \{u_1, u_2\}$. 
Note that $u_2$ is a $5$-vertex incident to $f_3$. 
If $f_4$ is a $6$-face, then since $v$ cannot be incident to $f_4$ twice, it must be that $u_1u_4$ is an edge by Proposition~\ref{degen6}.
Now, $u_1u_4u_2xu_3v$ is a $6$-cycle.
If $f_4$ is a $5$-face $zu_1vu_4w$, then $u_1u_2vu_4wz$ is a $6$-cycle, unless $u_2\in\{z, w\}$. 
If $u_2=z$, then $d(u_1)=2$ and if $u_2=w$, then $d(u_4)=2$. 
\end{proof}
}

\begin{claim}\label{4vertex-4444}
If a $4$-vertex $v$ is incident to only $4$-faces, then there are at least two incident $4$-faces that are not incident to special vertices. 
\end{claim}
\sout{
\begin{proof}
Let $u_1, u_2, u_3, u_4$ be the neighbors of $v$ in cyclic order so that $f_1$ is $u_1vu_2x$ for some $x$, $f_2$ is $u_2vu_3y$ for some $y$, and $f_3$ is $u_3vu_4z$ for some $z$.
Without loss of generality, either $y\not\in N[v]$ or $y=u_4$.
Since each vertex in $N[v]$ is incident to at least two $4$-faces, no vertex in $N[v]$ can be special.
If $y=u_1$, then by Proposition~\ref{adj4faces}, $z=u_2$. 
Thus, $f_1$ and $f_2$ are not incident to special vertices.
If $y\not\in N[v]$, then by Proposition~\ref{adj4faces}, $x=u_3$ and $z=u_2$.
Now, $f_1$, $f_2$, and $f_3$ are not incident to special vertices.
\end{proof}
}

For $i\in\{3, 4\}$, a vertex $v$ is {\it $i$-bad} if $d(v)=4$ and $v$ is incident to exactly $i$ $3$-faces.
A vertex is {\it bad} if it is either $3$-bad or $4$-bad; a vertex is {\it good} if it is neither bad nor special.
A face $f$ is {\it great} if $d(f)\geq 7$. 

\begin{figure}[h]
\centering
\includegraphics{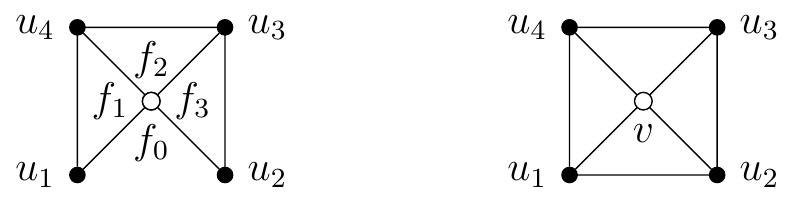}
  \caption{A $3$-bad vertex (left) and a $4$-bad vertex (right).
}\label{fig:tikz:badvertex}
\end{figure}

%

\begin{claim}\label{4bad}
A face that is not incident to a $4$-bad vertex $v$ but is adjacent to a $3$-face incident to $v$ is a great face. 
\end{claim}
\sout{
\begin{proof}
Let $u_1, u_2, u_3, u_4$ be the neighbors of $v$ in cyclic order as in Figure~\ref{fig:tikz:badvertex}.
By symmetry, we just need to show that a face $f$ that is adjacent to $u_1u_2$ but is not incident to $v$ is a great face.
Note that if either $u_1u_3$ or $u_2u_4$ is an edge, then $K^-_5\subseteq G[N[v]]$. 
If $f$ is a $3$-face $u_1u_2x$, then $x\in \{u_3, u_4\}$, otherwise $xu_1vu_4u_3u_2$ is a $6$-cycle.
Yet, if $x\in\{u_3, u_4\}$, then $K^-_5\subseteq G[N[v]]$.
If $f$ is a $4$-face $u_1u_2xy$, then $\{x, y\}=\{u_3, u_4\}$, otherwise $G$ has a $6$-cycle. 
Since $x\neq u_4$ and $y\neq u_3$, it must be that $x=u_3$ and $y=u_4$, which implies that $d(u_3)=3$. 
Also, $f$ cannot be a $5$-face since $f$ along with $v$ would form a $6$-cycle.
If $f$ is a $6$-face where $x, u_1, u_2, y$ are consecutive vertices on $f$, then, by Proposition~\ref{degen6} $(ii)$, either $xu_2\in E(G)$ or $u_1y\in E(G)$. 
In all cases, we get a $6$-cycle or a $K^-_5$. 
\end{proof}
}

\begin{claim}\label{3bad-triangle}
A face that is not incident to a $3$-bad vertex $v$ but is adjacent to a $3$-face incident to $v$ cannot be a $3$-face. 
\end{claim}
\sout{
\begin{proof}
Let $f_0, f_1, f_2, f_3$ be consecutive faces of $v$ and let $u_1, u_2, u_3, u_4$ be the neighbors of $v$ in cyclic order as in Figure~\ref{fig:tikz:badvertex}.
Note that $f_0$ cannot be a $3$-face, otherwise $v$ would be a $4$-bad vertex. 
Assume $f_2$ was adjacent to a $3$-face $u_3u_4x$ that is not $f_1, f_3$. 
If $x\not\in \{u_1, u_2\}$, then $xu_4u_1vu_2u_3$ is a $6$-cycle. 
If $x\in\{u_1, u_2\}$, then either $d(u_3)=3$ or $d(u_4)=3$. 

Without loss of generality, assume $f_3$ is adjacent to a $3$-face $u_2u_3x$ that is not $f_2$. 
If $x\not\in \{u_1, u_4\}$, then $xu_2vu_1u_4u_3$ is a $6$-cycle. 
If $x=u_4$, then $d(u_3)=3$. 
If $x=u_1$, then $K^-_5\subseteq G[N[v]]$.
\end{proof}
}

\begin{corollary}\label{3bad-bigface}
Each $3$-bad vertex $v$ is incident to either a great face or a degenerate $6$-face.
\end{corollary}
\sout{
\begin{proof}
Let $f_0$ be the face incident to $v$ that is not a $3$-face. 
By Claim~\ref{faces303}, $d(f_0)\geq 6$. 
If $f_0$ is a $6$-face, it must be a degenerate $6$-face, otherwise, $f_0$ is a great face. 
\end{proof}
}

\begin{corollary}\label{3bad-6}
If a $3$-bad vertex $v$ is incident to a degenerate $6$-face $f$, then a face that is not incident to $v$ but is adjacent to a face incident to $v$ must be a great face. 
\end{corollary}
\sout{
\begin{proof}
Let $u_1, u_2\in N(v)$ so that $u_1, v, u_2$ are consecutive vertices of $f$. 
Since $v$ cannot be incident to $f$ twice, by Proposition~\ref{degen6} $(ii)$, it must be that $u_1u_2\in E(G)$. 
The rest of the proof is identical to Claim~\ref{4bad}.
\end{proof}
}

\begin{corollary}\label{3bad-7}
If a $3$-bad vertex $v$ is incident to a great face $f$, then a face that is not incident to $v$ but is adjacent to a $3$-face incident to $v$ has length at least $6$.
\end{corollary}
\sout{
\begin{proof}
Let $f_0, f_1, f_2, f_3$ be consecutive faces of $v$ and let $u_1, u_2, u_3, u_4$ be the neighbors of $v$ in cyclic order as in Figure~\ref{fig:tikz:badvertex}.
Let $f$ be the face adjacent to $u_3u_4$ that is not $f_2$. 
By Claim~\ref{3bad-triangle}, $f$ cannot be a $3$-face. 
If $f$ is a $4$-face $u_3u_4xy$, then $\{x, y\}=\{u_1, u_2\}$, otherwise $G$ has a $6$-cycle. 
If either $y=u_2$ or $x=u_1$, then either $d(u_3)=2$ or $d(u_4)=2$. 
If $x=u_2$ and $y=u_1$, then $K^-_5\subseteq G[N[v]]$. 
Note that $f$ cannot be a $5$-face since $f$ along with $v$ would form a $6$-cycle.

Without loss of generality, let $f$ be the face adjacent to $u_2u_3$ that is not $f_3$. 
By Claim~\ref{3bad-triangle}, $f$ cannot be a $3$-face. 
If $f$ is a $4$-face $u_2xyu_3$ for some $x, y$, then $u_4vu_2xyu_3$ is a $6$-cycle, unless $u_4\in\{x, y\}$. 
Since $u_4=y$ implies $d(u_3)=3$, it must be that $u_4=x$. 
If $x=u_1$, then $K^-_5\subseteq G[N[v]]$, and if $x\neq u_1$, then $u_4xu_3u_2vu_1$ is a $6$-cycle. 
Note that $f$ cannot be a $5$-face since $f$ along with $v$ would form a $6$-cycle.
\end{proof}
}

\begin{corollary}\label{3bad-7-44}
Given a $3$-bad vertex $v$ incident to a great face, let $u_1, u_2, u_3, u_4$ be the neighbors of  $v$ in cyclic order so that $u_1vu_2$ is not a $3$-face. 
If $d(u_3)=d(u_4)=4$, then each face that is not incident to $v$ but is adjacent to a $3$-face incident to $v$ is great.
\end{corollary}
\sout{
\begin{proof}
Let $x\in N(u_2)\setminus\{u_1, v, u_3\}$ and let $y\in N(u_3)\setminus\{u_2, v, u_4\}$. 
For $i\in[3]$, let $f_i$ be the face that is incident to an edge $u_iu_{i+1}$ that is not a $3$-face.
By Corollary~\ref{3bad-7}, we know that $d(f_i)\geq 6$. 
Assume for the sake of contradiction that $f_i$ is a $6$-face for some $i\in[3]$.
Assume $i=1$.
By Proposition~\ref{degen6} $(ii)$, either $xu_1$ is an edge or $u_2$ is incident to $f_1$ twice. 
Yet, by Proposition~\ref{degen6} $(iv)$, $u_2$ cannot be incident to $f_1$ twice, so $xu_1$ must be an edge. 
If $x\neq u_4$, then $xu_1vu_4u_3u_2$ is a $6$-cycle.
If $x=u_4$, then $K^-_5\subseteq G[N[v]]$. 
By symmetry, this also solves the case when $i=3$. 

Assume $i=2$.
If $x=y$, then $x\not\in N[v]$.
Now, $u_1u_2xu_3u_4v$ is a $6$-cycle. 
If $x\neq y$, then by Proposition~\ref{degen6} $(iii)$, either $u_2$ or $u_3$ is incident to $f_2$ twice. 
In either case, this contradicts Proposition~\ref{degen6} $(iv)$.
\end{proof}
}

\begin{corollary}\label{bad-adjacent}
No two bad vertices are adjacent to each other.
\end{corollary}
\sout{
\begin{proof}
Follows from Claim~\ref{4bad} and Claim~\ref{3bad-triangle}.
\end{proof}
}

For $i\in\{1, 2\}$, a $5^+$-vertex $u$ is {\it $i$-responsible for} an adjacent bad vertex $v$ if $uv$ is incident to $i$ $3$-faces. 
A $5^+$-vertex $u$ is {\it responsible for} a bad vertex $v$ if $u$ is either $1$-responsible or $2$-responsible for $v$. 
A $4$-vertex $u$ is {\it responsible for} an adjacent bad vertex $v$ if $uv$ is incident to two $3$-faces. 
Note that a vertex might be responsible for several bad vertices, and several vertices might be responsible for the same bad vertex. 

\begin{corollary}\label{responsibleVertices}
Each vertex $v$ is responsible for at most $\lfloor{d(v)\over 2}\rfloor$ bad vertices. 
\end{corollary}
\sout{
\begin{proof}
If $v$ is responsible for a vertex $u$,
one of the two faces incident to the edge $vu$ must be a $3$-face $vux$. 
By Corollary~\ref{bad-adjacent}, $x$ cannot be a bad vertex.
By Claim~\ref{4bad} and Claim~\ref{3bad-triangle}, the face incident to $xv$ that is not $xvu$ has length at least $6$, and this finishes the proof. 
\end{proof}
}

\begin{corollary}\label{2responsibleVertices}
Each vertex $v$ is $2$-responsible for at most $\lfloor {d(v)\over 3}\rfloor$ bad vertices.
\end{corollary}
\sout{
\begin{proof}
Let $x_1, \ldots, x_{d(x)}$ be the neighbors of $v$ in cyclic order. 
If $v$ is $2$-responsible for $x_i$, then both faces incident to the edge $vx_i$ must be $3$-faces. 
By Claim~\ref{4bad} and Claim~\ref{3bad-triangle}, the face incident to $v, x_{i+1}, x_{i+2}$ cannot be a $3$-face, thus, $v$ cannot be $2$-responsible for $x_{i+1}$ and $x_{i+2}$. 
By the same argument, $v$ cannot be $2$-responsible for $x_{i-1}, x_{i-2}$. 
\end{proof}
}

\section{Discharging}

Recall that an embedding of $G$ was fixed, and let $F(G)$ be the set of faces of $G$.
In this section, we will prove that $G$ cannot exist by assigning an \emph{initial charge} $\mu(z)$ to each $z\in V(G) \cup F(G)$, and then applying a discharging procedure to end up with {\it final charge} $\mu^*(z)$ at $z$.
We prove that the final charge has positive total sum, whereas the initial charge sum is at most zero.
The discharging process will preserve the total charge sum, and hence we find a contradiction to conclude that $G$ does not exist.

For each vertex $v\in V(G)$, let $\mu(v)=d(v)-6$, and for each face $f\in F(G)$, let $\mu(f)=2d(f)-6$. 
The total initial charge is at most zero since
\begin{align*}
\sum_{z\in V(G)\cup F(G)} \mu(z)
	=\sum_{v\in V(G)} (d(v)-6)+\sum_{f\in V(F)} (2d(f)-6) 
	=6|E(G)|-6|V(G)|-6|F(G)|
	\leq 0.
\end{align*}
The final equality holds by Euler's formula.

The rest of this section will prove that the sum of the final charge after the discharging phase is positive. 


Recall that a $4$-vertex $v$ is {\it special} if $v$ is incident to a $4$-face and exactly two $3$-faces.
A $4$-vertex $v$ is {\it bad} if is incident to three or four $3$-faces; a vertex is {\it good} if it is neither bad nor special.
For $i\in\{1, 2\}$, a $5^+$-vertex $u$ is {\it $i$-responsible for} an adjacent bad vertex $v$ if $uv$ is incident to $i$ $3$-faces. 
A $5^+$-vertex $u$ is {\it responsible for} a vertex $v$ if $u$ is either $1$-responsible or $2$-responsible for $v$. 
A $4$-vertex $u$ is {\it responsible for} an adjacent bad vertex $v$ if $uv$ is incident to two $3$-faces. 
A face $f$ is {\it great} if $d(f)\geq 7$. 

Here are the discharging rules:

\begin{enumerate}[(R1)]
\item Each $4$-face sends charge 
$1$ to each incident special vertex, 
${1\over 5}$ to each incident $5^+$-vertex, 
and distributes its remaining initial charge uniformly to each incident non-special $4$-vertex.


\item Each $5$-face sends charge 
${4\over 7}$ to each incident $5^+$-vertex 
and distributes its remaining initial charge uniformly to each incident $4$-vertex. 

\item Each $6^+$-face distributes its initial charge uniformly to each incident vertex. 

\item Each good $4$-vertex $u$ sends its excess charge to each vertex $v$ where $u$ is responsible for $v$.

\item Each $5^+$-vertex $u$ sends charge $1$ to each vertex $v$ where $u$ is $2$-responsible for $v$. 

\item Each $5^+$-vertex $u$ sends charge $2\over 7$ to each vertex $v$ where $u$ is $1$-responsible for $v$. 

\end{enumerate}

\begin{figure}[h]
\centering
\includegraphics{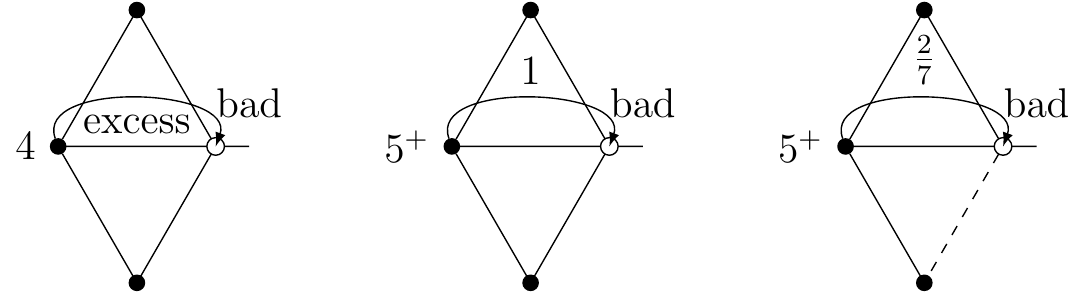}
  \caption{Discharging Rules}
  \label{fig:tikz:fig-rules}
\end{figure}

We will first show that each $4$-face has nonnegative final charge.
It is trivial that each $6^+$-face has nonnegative final charge. 
Then, we will show that each vertex has nonnegative final charge.
Moreover, we will show that each bad vertex and each $5^+$-vertex that is not adjacent to a bad vertex has positive final charge.

\begin{claim}\label{4face}
Each $4$-face $f$ has nonnegative final charge.
Moreover, $f$ sends charge at least $3\over 5$ to each incident $4$-vertex if $f$ is not incident to a special vertex, and $f$ sends charge at least $2\over 5$ to each incident non-special $4$-vertex if $f$ is incident to a special vertex.
\end{claim}
\sout{
\begin{proof}
By Claim~\ref{4face-oneSpecial}, $f$ is incident to at most one special vertex. 
By Lemma~\ref{reducible}, there are at most three vertices of degree $4$ incident to $f$. 
Since ${1\over 5}<{2\over 5}<{3\over 5}$, the worst case is when $f$ has many incident $4$-vertices.
If $f$ is not incident to a special vertex, then $\mu^*(f)\geq 2-3\cdot{3\over 5}-{1\over 5}\geq 0$.
If $f$ is incident to a special vertex, then $\mu^*(f)\geq 2-1-2\cdot{2\over 5}-{1\over 5}\geq 0$.
\end{proof}
}

\begin{claim}\label{5face}
Each $5$-face $f$ has nonnegative final charge.
Moreover, $f$ sends charge at least $6\over 7$ to each incident $4$-vertex if $f$ is incident to a $5^+$-vertex, and $f$ sends charge at least $4\over 5$ to each incident $4$-vertex if $f$ is not incident to a $5^+$-vertex.
\end{claim}
\sout{
\begin{proof}
Since ${4\over 7}<{4\over 5}<{6\over 7}$, the worst case is when $f$ has many incident $4$-vertices.
If $f$ is incident to a $5^+$-vertex, then $\mu^*(f)\geq 4-4\cdot{6\over 7}-{4\over 7}\geq 0$.
If $f$ is not incident to a $5^+$-vertex, then $\mu^*(f)\geq 4-5\cdot{4\over 5}\geq 0$.
\end{proof}
}

Note that each (degenerate) $6$-face sends charge $1$ to each incident vertex, 
and each great face $f$ sends charge ${\mu(f)\over d(f)}={2d(f)-6\over d(f)}\geq {8\over 7}$ to each incident vertex. 

\begin{claim}\label{6vertex}
Each $6^+$-vertex $v$ has positive final charge. 
Moreover, if $v$ is not adjacent to a bad vertex, then it has positive final charge.
\end{claim}
\sout{
\begin{proof}
By Claim~\ref{4bad}, Corollary~\ref{3bad-6}, and Corollary~\ref{3bad-7}, for each vertex $v$ is responsible for, there exist two faces of length at least $6$ incident to $v$ that will each send charge at least $1$ to $v$.
If $v$ is responsible for $r$ vertices, then, $\mu^*(v)\geq {2\cdot r\over 2}-1\cdot r\geq 0$.

If $v$ is not adjacent to a bad vertex, then $v$ is not responsible for any vertex. 
Also $v$ cannot be incident to only $3$-faces since this would create a $6$-cycle. 
Now, since $v$ is incident to a $4^+$-face, $v$ has positive final charge. 
\end{proof}
}

\begin{claim}\label{5vertex}
Each $5$-vertex $v$ has nonnegative final charge. 
Moreover, if $v$ is not adjacent to a bad vertex, then it has positive final charge.
\end{claim}
\sout{
\begin{proof}
By Claim~\ref{responsibleVertices}, $v$ is responsible for at most two vertices, and
by Claim~\ref{2responsibleVertices}, $v$ is $2$-responsible for at most one vertex. 
If $v$ is $2$-responsible for a vertex and $1$-responsible for a vertex, then there must be two great faces incident to $v$ by Corollary~\ref{faces3303}.
Thus, $\mu^*(v)\geq -1-1-{2\over 7}+2\cdot{8\over 7}\geq0$. 
If $v$ is $2$-responsible for a vertex and is not $1$-responsible for any vertex, then $v$ is incident to at least two $6^+$-faces by Claim~\ref{4bad}, Corollary~\ref{3bad-6}, and Corollary~\ref{3bad-7}.
Thus, $\mu^*(v)\geq -1-1+2\cdot1\geq 0$. 

If $v$ is not $2$-responsible for any vertex, then $v$ is $1$-responsible for at most two vertices.
If $v$ is $1$-responsible for at least one vertex, then $v$ is incident to at least two $6^+$-faces, by Claim~\ref{4bad}, Corollary~\ref{3bad-6}, and Corollary~\ref{3bad-7}.
Thus, $\mu^*(v)\geq -1-2\cdot{2\over 7}+2\cdot 1>0$.

The only case left is when $v$ is not responsible for any vertex.
If there are three consecutive faces $f_1, f_2, f_3$ where $d(f_1)=d(f_3)=3\neq d(f_2)$, then $d(f_2)\geq 6$ by Claim~\ref{faces303}. 
Since the other two faces cannot be both $3$-faces, $\mu^*(v)\geq -1+1+{1\over 5}>0$.

If there are consecutive faces $f_1, f_2, f_3, f_4$ where $d(f_1)=d(f_2)=3$ and $d(f_3)=4$, then $d(f_4)\geq 6$ by Claim~\ref{faces3340}. 
Thus, $\mu^*(v)\geq -1+1+{1\over 5}>0$.
Thus, given consecutive faces $f_0, f_1, f_2, f_3$ where $d(f_1)=d(f_2)=3$, then both $d(f_0), d(f_3)\geq 5$.
Thus, $\mu^*(v)\geq -1+2\cdot{4\over 7}>0$.

Now, let $v$ be incident to at most one $3$-face.
If $v$ is incident to one $3$-face, then by Claim~\ref{faces4434}, there exists a $5^+$-face incident to $v$. 
Thus, $\mu^*(v)\geq -1+{4\over 7}+3\cdot{1\over 5}>0$.
Note that $v$ cannot be incident to only $4$-faces by Claim~\ref{faces44444}.
Thus, $v$ is incident to at least one $5^+$-face and at four $4^+$-faces. 
Thus, $\mu^*(v)\geq -1+4\cdot{2\over 5}+{4\over 5}>0$.

Note that if $v$ is not adjacent to a bad vertex, then $v$ is not responsible for any vertex.
Also $v$ cannot be incident to only $3$-faces since this would create a $6$-cycle. 
Now, since $v$ is incident to a $4^+$-face, $v$ has positive final charge. 
\end{proof}
}

\begin{claim}\label{good4vertex}
Each good $4$-vertex $v$ has nonnegative final charge.
\end{claim}
\sout{
\begin{proof}
Note that $v$ is incident to at most two $3$-faces, otherwise $v$ is a bad vertex. 
If $v$ is incident to two $3$-faces that are not adjacent to each other, then the other two faces have length at least $6$ by Claim~\ref{faces303}.
Thus, $\mu^*(v)\geq -2+2\cdot 1\geq 0$.
If $v$ is incident to two $3$-faces that are adjacent to each other and $v$ is responsible for at least one vertex, then the other two faces must be $6^+$-faces by Claim~\ref{4bad}, Corollary~\ref{3bad-6}, and Corollary~\ref{3bad-7}.
Thus, $\mu^*(v)\geq -2+2\cdot1\geq 0$.

Assume $v$ is incident to two $3$-faces that are adjacent to each other and $v$ is not responsible for any vertex. 
If $v$ is incident to a $4$-face, then $v$ is a special vertex. 
If $v$ is incident to a $5$-face, then by Claim~\ref{4vertex-335}, $v$ is also incident to a $7^+$-face and the $5$-face is incident to a $5^+$-vertex.
Thus, $\mu^*(v)\geq -2+{8\over 7}+{6\over 7}\geq 0$.
If $v$ is incident to $6^+$-faces, then, $\mu^*(v)\geq -2+2\cdot 1\geq 0$.
%

Assume $v$ is incident to one $3$-face $f_1$, where $f_1, f_2, f_3, f_4$ are consecutive faces of $v$.
If $d(f_3)\geq 5$, and the other faces are $4$-faces, then by Claim~\ref{faces3454}, $\mu^*(v)\geq -2+{4\over 5}+2\cdot{3\over 5}\geq0$.
If $d(f_3)\geq 5$, and the other faces are not both $4$-faces, then $\mu^*(v)\geq -2+{4\over 5}+{4\over 5}+{2\over 5}\geq0$.
If $d(f_3)=4$, then by Claim~\ref{4vertex-304}, either $\mu^*(v)\geq -2+{2\over 5}+1+{4\over 5}>0$ or $\mu^*(v)\geq -2+{3\over 5}+{2\over 5}+1\geq0$.

Assume $v$ is incident to only $4^+$-faces.
If a $5^+$-face is incident to $v$, then $\mu^*(v)\geq -2+{4\over 5}+3\cdot{2\over 5}\geq0$.
If $v$ is incident to only $4$-faces, then by Claim~\ref{4vertex-4444}, at least two of the $4$-faces give charge at least $3\over 5$. 
Thus, $\mu^*(v)\geq -2+2\cdot{3\over 5}+2\cdot{2\over 5}\geq0$.
\end{proof}
}

\begin{claim}
Each special vertex $v$ has nonnegative final charge.
\end{claim}
\sout{
\begin{proof}
A special vertex $v$ is incident to two $3$-faces and a $4$-face.
By Claim~\ref{faces3340}, the fourth face must be a $6^+$-face.
Thus, $\mu^*(v)\geq -2+1+1\geq 0$.
\end{proof}
}

\begin{claim}
Each $3$-bad vertex $v$ incident to a great face has positive final charge.
\end{claim}
\sout{
\begin{proof}
Let $u_1, u_2, u_3, u_4$ be the neighbors of $v$ in cyclic order so that $u_1vu_4$ is not a $3$-face. 
According to Lemma~\ref{reducible}, either $d(u_2)=d(u_3)=4$ and $d(u_1), d(u_4)\geq 5$ or $d(u_i)\geq 5$ for some $i\in\{2, 3\}$. 
In the former, $u_2, u_3$ sends charge at least $2\over 7$ since they are incident to two $7^+$-faces by Corollary~\ref{3bad-7-44}.
Thus, $\mu^*(v)\geq -2+{8\over 7}+2\cdot{2\over 7}+2\cdot{2\over 7}>0$. 
In the latter, $\mu^*(v)\geq -2+{8\over 7}+1>0$. 
\end{proof}
}

\begin{claim}
Each $3$-bad vertex $v$ incident to a degenerate $6$-face has positive final charge.
\end{claim}
\sout{
\begin{proof}
Let $u_1, u_2, u_3, u_4$ be the neighbors of $v$ in cyclic order so that $u_1, v, u_4$ is not a $3$-face.
For $i\in\{2, 3\}$, if $d(u_i)=4$, then it sends charge at least $2\over 7$ since it is incident to two $7^+$-faces by Corollary~\ref{3bad-6}.
According to Lemma~\ref{reducible}, either $d(u_2)=d(u_3)=4$ and $d(u_1), d(u_4)\geq 5$, or $d(u_i)\geq 4$ and $d(u_j)\geq 5$ for $\{i, j\}=\{2, 3\}$.
In the former case, $\mu^*(v)\geq -2+1+2\cdot{2\over 7}+2\cdot{2\over 7}>0$. 
In the latter case, $\mu^*(v)\geq -2+1+1+{2\over 7}>0$.
\end{proof}
}

\begin{claim}
Each $4$-bad vertex $v$ has positive final charge.
\end{claim}
\sout{
\begin{proof}
According to Lemma~\ref{reducible}, at least two vertices in $N(v)$ must have degree at least $5$. 
Note that each $4$-vertex in $N(v)$ sends charge $2\over 7$ since they are not incident to $6$-faces by Claim~\ref{4bad}.
Thus, $\mu^*(v)\geq -2+2\cdot 1+2\cdot{2\over 7}>0$.
\end{proof}
}

Since each bad vertex has positive final charge, there are no bad vertices.
Since each $5^+$-vertex $v$ that is not adjacent to bad vertices has positive final charge, it must be the case that $G$ is $4$-regular.
Since there is no $K^-_5$, there is no $K_5$, and by Theorem~\ref{deg-choos}, we know that $G$ is $4$-choosable, which contradicts the assumption that $G$ is a counterexample.

\section{Sharpness Examples}

In this section, we show that Theorem~\ref{result} is sharp by showing that we must forbid both $K^-_5$ and $6$-cycles. 
It is worth mentioning that both infinite families of graphs is embeddable on any surface, orientable or non-orientable, except the plane and projective plane. 
Note that Theorem~\ref{noC6} disproves a conjecture in \cite{2010CaWaZh}.

\begin{figure}[h]
\centering
\includegraphics{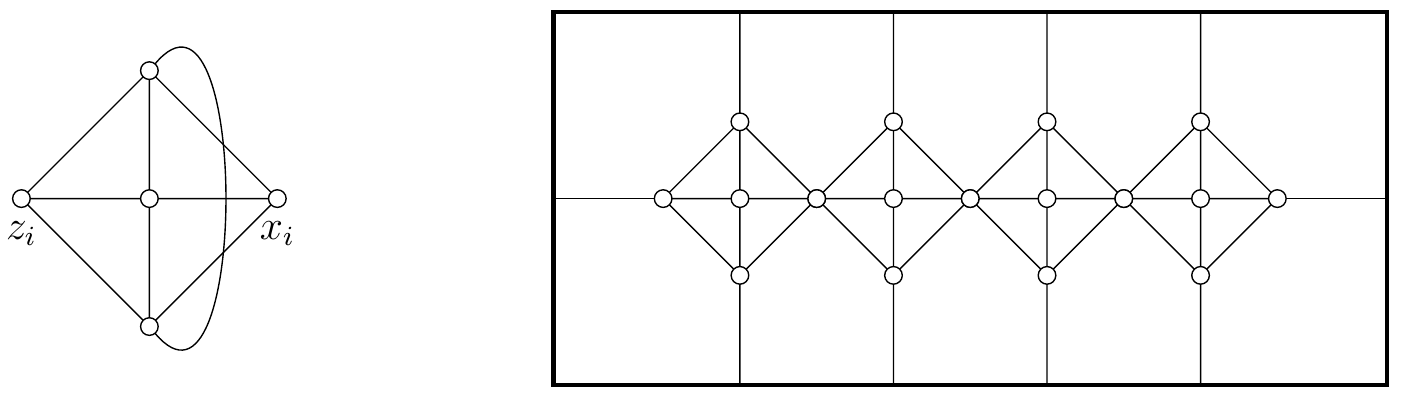}
  \caption{The graphs $H_i$ (left) and an embedding of $G_4$ on the torus (right).}
  \label{fig:tikz:fig-counterexampleC6}
\end{figure}

\begin{theorem}\label{noC6}
For each $k\geq 6$, there exists an infinite family of toroidal graphs without $\ell$-cycles for any $6\leq \ell\leq k$ with chromatic number $5$. 
\end{theorem}
\begin{proof}
Let $H_i$ be a complete graph on $5$ vertices minus an edge where $x_i$ and $z_i$ are the vertices of degree $3$. 
Create $G_s$ in the following way: given $s$ copies of $H_i$, identify $x_i$ and $z_{i+1}$ for $i\in[s-1]$, and also add the edge $x_sz_1$. 

Let $s\geq \lceil{k\over 2}\rceil$ and consider $G_s$. 
It is easy to check that any $4$-coloring of $G_s$ must assign the same color to all the identified vertices as well as $x_s, z_1$, which is a contradiction since $x_sz_1$ is an edge. 
This shows $G_s$ is not $4$-colorable, which further implies that it is not $4$-choosable. 
For each cycle in $G_s$, if it uses the the edge $x_sz_1$, then it must have length at least $2s+1$, which is at least $k+1$. 
All other cycles are contained within a copy of $H_i$, and has length at most $5$. 
It is easy to check that there is no $K_5$ and that there is a $5$-coloring of $G_s$.
Note that $G_s$ is toroidal, as seen in Figure~\ref{fig:tikz:fig-counterexampleC6}.
\end{proof}

\begin{figure}[h]
\centering
\includegraphics{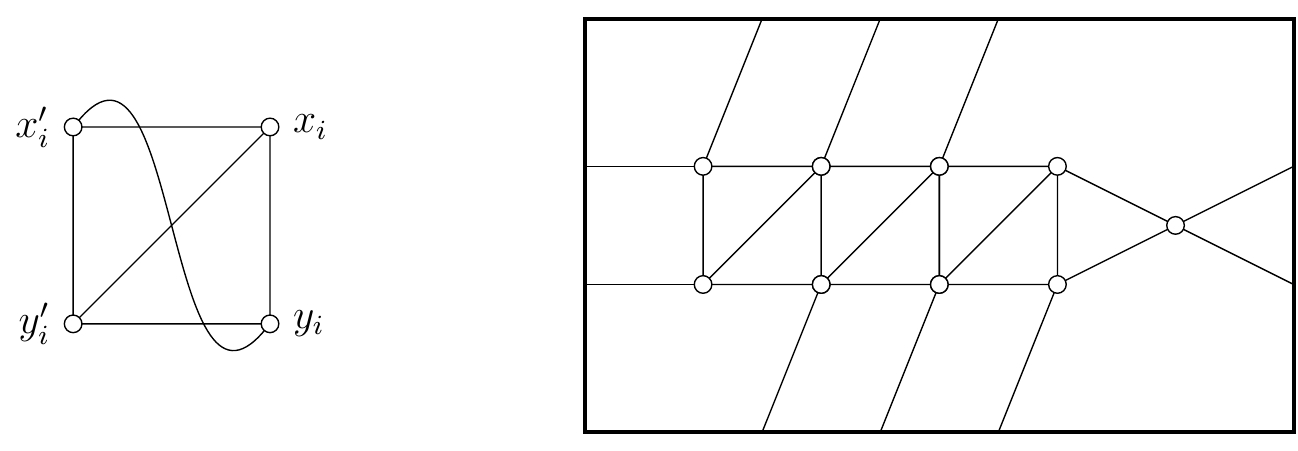}
  \caption{The graphs $H'_i$ (left) and an embedding of $G'_1$ on the torus (right).}
  \label{fig:tikz:fig-counterexampleK5e}
\end{figure}

\begin{theorem}\label{noK5e}
There exists an infinite family of hamiltonian toroidal graphs without $K^-_5$ with chromatic number $5$.
\end{theorem}
\begin{proof}
Let $H'_i$ be a complete graph on $4$ vertices where $x_i$, $x'_i$, $y_i$, $y'_i$ are the vertices of $H_i$.
Create $G'_s$ in the following way: given $2s+1$ copies of $H'_i$, identify $x_i$ with $x'_{i+1}$ and identify $y_i$ with $y'_i$ for $i\in[2s]$, and also add a vertex $z$ and the edges $zx'_1$, $zy'_1$, $zx'_{2s+1}$, and $zy'_{2s+1}$.

It is easy to check that any $4$-coloring of $G'_s$ must assign different colors to the neighbors of $z$, which implies that $G'_s$ is not $4$-colorable; this further implies that it is not $4$-choosable. 
It is easy to check that $G'_s$ is hamiltonian and there is no $K^-_5$.
Note that $G'_s$ is toroidal, as seen in Figure~\ref{fig:tikz:fig-counterexampleK5e}.
\end{proof}

\section{Acknowledgments}

The author thanks Xuding Zhu for suggesting the statement of Theorem~\ref{result} and Alexandr V. Kostochka for improving the readability of the paper.

\bibliographystyle{plain}
\bibliography{Torus_NoK5eC6_4choos}

\begin{thebibliography}{10}

\bibitem{1999BoMoSt}
T.~B{\"o}hme, B.~Mohar, and M.~Stiebitz.
\newblock Dirac's map-color theorem for choosability.
\newblock {\em J. Graph Theory}, 32(4):327--339, 1999.

\bibitem{2013Bo}
O.V. Borodin.
\newblock Colorings of plane graphs: A survey.
\newblock {\em Discrete Mathematics}, 313(4):517 -- 539, 2013.

\bibitem{2010CaWaZh}
L. Cai, W. Wang, and X. Zhu.
\newblock Choosability of toroidal graphs without short cycles.
\newblock {\em J. Graph Theory}, 65(1):1--15, 2010.

\bibitem{2009Fa}
B. Farzad.
\newblock Planar graphs without 7-cycles are 4-choosable.
\newblock {\em SIAM J. Discrete Math.}, 23(3):1179--1199, 2009.

\bibitem{00Ko}
A. Kostochka.
\newblock Private communication.

\bibitem{1999LaXuLi}
P. Lam, B. Xu, and J. Liu.
\newblock The {$4$}-choosability of plane graphs without {$4$}-cycles.
\newblock {\em J. Combin. Theory Ser. B}, 76(1):117--126, 1999.

\bibitem{1994Th}
C. Thomassen.
\newblock Every planar graph is {$5$}-choosable.
\newblock {\em J. Combin. Theory Ser. B}, 62(1):180--181, 1994.

\bibitem{1993Vo}
M. Voigt.
\newblock List colourings of planar graphs.
\newblock {\em Discrete Math.}, 120(1-3):215--219, 1993.

\bibitem{2001WaLi}
W. Wang and K.-W. Lih.
\newblock The 4-choosability of planar graphs without 6-cycles.
\newblock {\em Australas. J. Combin.}, 24:157--164, 2001.

\bibitem{2002WaLi}
W. Wang and K.-W. Lih.
\newblock Choosability and edge choosability of planar graphs without five
  cycles.
\newblock {\em Appl. Math. Lett.}, 15(5):561--565, 2002.

\bibitem{00Zh}
X. Zhu.
\newblock Private communication.

\end{thebibliography}

%
%

\end{document}